\documentclass[10pt]{amsart}
\usepackage{amsmath,amssymb,latexsym,hyperref,url}

\usepackage{textcomp}		

\newtheorem*{lemma*}{Lemma}

\newcommand{\R}{\mathbb{R}}

\begin{document}

\title{What's in \emph{YOUR} wallet?}
\author{Lara Pudwell}
\address{
	Department of Mathematics and Statistics\\
	Valparaiso University\\
	Valparaiso, Indiana 46383, USA
}
\author{Eric Rowland}
\address{
	LaCIM \\
	University of Qu\'ebec at Montr\'eal \\
	Montr\'eal, QC H2X 3Y7, Canada
}
\curraddr{
	Department of Mathematics \\
	University of Liege \\
	4000 Li\`ege, Belgium
}
\date{May 24, 2015}


\maketitle

While you probably associate the title of this paper with credit card commercials, we suggest it is actually an invitation to some pretty interesting mathematics.  Every day, when customers spend cash for purchases, they exchange coins.  There are a variety of ways a spender may determine which coins from their wallet to give a cashier in a transaction, and of course a given spender may not use the same algorithm every time.
In this paper, however, we make some simplifying assumptions so that we can provide an answer to the question `What is the expected number of coins in your wallet?'.

Of course, the answer depends on where you live!
A \emph{currency} is a set of denominations.
We'll focus on the currency consisting of the common coins in the United States, which are the quarter (25 cents), dime (10 cents), nickel (5 cents), and penny (1 cent).
However, we invite you to grab your passport and carry out the computations for other currencies.
Since we are interested in distributions of coins, we will consider prices modulo $100$ cents, in the range $0$ to $99$.

The contents of your wallet largely depend on how you choose which coins to use in a transaction.
We'll address this shortly, but let's start with a simpler question.
How does a cashier determine which coins to give you as change when you overpay?
If you are due $30$ cents, a courteous cashier will not give you $30$ pennies.
Generally the cashier minimizes the number of coins to give you, which for $30$ cents is achieved by a quarter and nickel.
Therefore let's make the following assumptions.
\begin{enumerate}
\item\label{1} The fractional parts of prices are distributed uniformly between $0$ and $99$ cents.
\item\label{2} Cashiers return change using the fewest possible coins.
\end{enumerate}

Is there always a unique way to make change with the fewest possible coins?
It turns out that for every integer $n \geq 0$ (not just $0 \leq n \leq 99$) there is a unique integer partition of $n$ into parts $25$, $10$, $5$, and $1$ that minimizes the number of parts.
And this is what the cashier gives you, assuming there are enough coins of the correct denominations in the cash register to cover it, which is a reasonable assumption since a cashier with only $3$ quarters, $2$ dimes, $1$ nickel, and $4$ pennies can give change for any price that might arise.

The cashier can quickly compute the minimal partition of an integer $n$ into parts $d_1, d_2, \dots, d_k$ using the \emph{greedy algorithm} as follows.
To construct a partition of $n = 0$, use the empty partition $\{\}$.
To construct a partition of $n \geq 1$, determine the largest $d_i$ that is less than or equal to $n$, and add $d_i$ to the partition; then recursively construct a partition of $n - d_i$ into parts $d_1, d_2, \dots, d_k$.
For example, if $37$ cents is due, the cashier first takes a quarter from the register; then it remains to make change for $37 - 25 = 12$ cents, which can most closely be approximated (without going over) by a dime, and so on.
The greedy algorithm partitions $37$ into $\{25,10,1,1\}$.\footnote{Maurer~\cite{Maurer} interestingly observes that before the existence of electronic cash registers, cashiers typically did not use the greedy algorithm but instead counted \emph{up} from the purchase price to the amount tendered --- yet still usually gave change using the fewest coins.}

We remark that for other currencies the greedy algorithm does not necessarily produce partitions of integers into fewest parts.
For example, if the only coins in circulation were a $4$-cent piece, a $3$-cent piece, and a $1$-cent piece, the greedy algorithm makes change for $6$ cents as $\{4, 1, 1\}$, whereas $\{3, 3\}$ uses fewer coins.
In general it is not straightforward to tell whether a given currency lends itself to minimal partitions under the greedy algorithm.
Indeed, there is substantial literature on the subject~\cite{Adamaszek--Adamaszek, Cai, Chang--Korsh, Chang--Gill, Kozen--Zaks, Magazine--Nemhauser--Trotter}
and at least one published false ``theorem''~\cite{Jones, Maurer}.
Pearson~\cite{Pearson} gave the first polynomial-time algorithm for determining whether a given currency has this property.

As for spending coins, the simplest way to spend coins is to not spend them at all.
A \emph{coin keeper} is a spender who never spends coins.
Sometimes when you're traveling internationally it's easier to hand the cashier a big bill than try to make change with foreign coins.  Or maybe you don't like making change even with domestic coins, and at the end of each day you throw all your coins into a jar.
In either case, you will collect a large number of coins.
What is the distribution?

It is easy to compute the change you receive if you spend no coins in each of the $100$ possible transactions corresponding to prices from $0$ to $99$ cents.
Since we assume these prices appear with equal likelihood, to figure out the long-term distribution of coins in a coin keeper's collection, we need only tally the coins of each denomination.
A quick computer calculation shows that the coins received from these 100 transactions total 150 quarters, 80 dimes, 40 nickels, and 200 pennies.  In other words, a coin keeper's stash contains 31.9\%~quarters, 17.0\%~dimes, 8.5\%~nickels, and 42.6\%~pennies.

What's in the country's wallet?
The coin keeper's distribution looks quite different from that of coins actually manufactured by the U.S. mint.
In 2014, the U.S. government minted $1580$ million quarters, $2302$ million dimes, $1206$ million nickels, and $8146$ million pennies~\cite{coin production} --- that's
11.9\%~quarters, 17.4\%~dimes, 9.1\%~nickels, and 61.6\%~pennies.

Fortunately, most of us do not behave as coin keepers.
So let us move on to spenders who are not quite so lazy.

\section*{Markov chains}

When you pay for your weekly groceries, the state of your wallet as you leave the store depends only on
\begin{itemize}
\item[(i)]
the state of your wallet when you entered the store,
\item[(ii)]
the price of the groceries, and
\item[(iii)]
the algorithm you use to determine how to pay for a given purchase with a given wallet state.
\end{itemize}
So what we have is a \emph{Markov chain}.

A Markov chain is a system in which for all $t \geq 1$ the probability of being in a given state at time $t$ depends only on the state of the system at time $t - 1$.
Here time is discrete, and at every time step a random event occurs to determine the new state of the system.
The main defining feature of a Markov chain is that the probability of the system being in a given state does not depend on the system's history before time $t - 1$.
For us, the system is the spender's wallet, and the random event is the purchase price.

Let $S = \{s_1, s_2, \dots\}$ be the set of possible states of your wallet.
A Markov chain with finitely many states has a $|S| \times |S|$ \emph{transition matrix} $M$ whose entry $m_{ij}$ is the probability of transitioning to $s_j$ if the current state of the system is $s_i$.
By assumption, $m_{ij}$ is independent of the time at which $s_i$ occurs.
If $v = \left(v_1, v_2, \dots, v_{|S|}\right)$ is a vector whose entry $v_i$ is the probability of your wallet being in state $s_i$ initially, then $v M$ is a vector whose $i$th entry is the probability of the wallet being in state $s_i$ after one step.

The long-term behavior of your wallet is therefore given by $v M^n$ for large $n$.
If the limit $p = \lim_{n \to \infty} v M^n$ exists, then there is a clean answer to a question such as `What is the expected number of coins in your wallet?', since the $i$th entry $p_i$ of $p$ is the long-term probability that your wallet is in state $s_i$.
Moreover, if the limit is independent of the initial distribution $v$, then $p$ is not just the long-term distribution for your wallet; it's the long-term distribution for any wallet using the same spending strategy.

Supposing for the moment that $p$ exists, how can we compute it?
The limiting probability distribution does not change under multiplication by $M$ (because otherwise it's not the limiting probability distribution), so $p M = p$.
In other words, $p$ is a left eigenvector of $M$ with eigenvalue $1$.
There may be many such eigenvectors, but we know additionally that $p_1 + p_2 + \cdots + p_{|S|} = 1$, which may be enough information to uniquely determine the entries of $p$.

It turns out that, under reasonable spending assumptions, the Perron--Frobenius theorem guarantees the existence and uniqueness of $p$.
We need two conditions on the Markov chain --- irreducibility and aperiodicity.
A Markov chain is \emph{irreducible} if for any two states $s_i$ and $s_j$ there is some integer $n$ such that the probability of transitioning from $s_i$ to $s_j$ in $n$ steps is nonzero.
That is, each state is reachable from each other state, so the state space can't be broken up into two nonempty sets that don't interact with each other in the long term.
For each Markov chain we consider, irreducibility follows from assumptions \eqref{1}--\eqref{2} above and details of the particular spending strategy (for example, assumptions \eqref{3}--\eqref{4} below).

The other condition is aperiodicity.
A Markov chain is \emph{periodic} (i.e., not aperiodic) if there is some state $s_i$ such that any transition from $s_i$ to itself occurs in a multiple of $k>1$ steps.
If a wallet is in state $s_i$, then the transaction with price $0$ causes the wallet to transition to $s_i$, so wallet Markov chains are aperiodic.
Therefore the Perron--Frobenius theorem implies that $p$ exists and that $p$ is the dominant left eigenvector of the matrix $M$, corresponding to the eigenvalue $1$.
From $p$ we can compute all sorts of statistics.
For example, the expected number of coins in the wallet is
\[
	\sum_{i=1}^{|S|} p_i |s_i|.
\]
The expected total value of the wallet, in cents, is
\[
	\sum_{i=1}^{|S|} p_i \sigma(s_i),
\]
where $\sigma(s_i)$ is the sum of the elements in $s_i$.

\section*{Spending strategies}

Now that we understand the mechanics of Markov chains, we can use them to study various models of a spender's behavior.
Unlike the cashier, the spender has a limited supply of coins.
When the supply is limited, the greedy algorithm does not always make exact change.
For example, if you're trying to come up with $30$ cents and your wallet state is $\{25, 10, 10, 10\}$ then the greedy algorithm fails to identify $\{10, 10, 10\}$.

Moreover, the spender will not always be able to make exact change.
Since our spender does not want to accumulate arbitrarily many coins (unlike the coin keeper), let's first consider the \emph{minimalist spender}, who spends coins so as to minimize the number of coins in their wallet after each transaction.

\subsection*{The minimalist spender}

Of course, one way to be a minimalist spender is to curtly throw all your coins at the cashier and ask them to give you change (greedily).
Sometimes this can result in clever spending; for example if you have $\{10\}$ and are charged $85$ cents, then you'll end up with $\{25\}$.
However, in other cases this is socially uncouth; if you have $\{1, 1, 1, 1\}$ and are charged $95$ cents, then the cashier will hand you back $\{5, 1, 1, 1, 1\}$, which contains the four pennies you already had.
With some thought, you can avoid altercations by not handing the cashier any coins they will hand right back to you.

In any case, if a minimalist spender's wallet has value $n$ cents and the price is $c$ cents, then the state of the wallet after the transaction will be a minimal partition of $n - c \bmod 100$.
Since there is only one such minimal partition, this determines the minimalist spender's wallet state.
There are $100$ possible wallet states, one for each integer $0 \leq n \leq 99$.
By assumption~\eqref{1}, the probability of transitioning from one state to any other state is $1/100$, so no computation is necessary to determine that each state is equally likely in the long term.
The expected number of coins in the minimalist spender's wallet is therefore $\frac{1}{100} \sum_{i=1}^{100} |s_i| = 4.7$, and the expected total value of the wallet is $\frac{1}{100} \sum_{n=0}^{99} n = 49.5$ cents.
Counting occurrences of each denomination in the $100$ minimal partitions of $0 \leq n \leq 99$ shows that the expected number of quarters is $1.5$; the expected numbers of dimes, nickels, and pennies are $0.8$, $0.4$, and $2$.

Intuitively, one would expect the minimalist's strategy to result in the lowest possible expected number of coins.
Indeed this is the case; let $g(n)$ be the number of coins in the greedy partition of $n$.
Fix a spending strategy that yields an irreducible, aperiodic Markov chain.
Let $e(n)$ be the long-term expected number of coins in the spender's wallet, conditional on the (long-term) total value being $n$ cents.
Since $g(n)$ is the minimum number of coins required to have exactly $n$ cents, $e(n) \geq g(n)$ for all $0 \leq n \leq 99$.
Since the price $c$ is uniformly distributed, the total value $n$ is uniformly distributed, and therefore the long-term expected number of coins is
\[
	\frac{1}{100} \sum_{n=0}^{99} e(n) \geq \frac{1}{100} \sum_{n=0}^{99} g(n) = \frac{47}{10}.
\]

However, the minimalist spender's behavior is not very realistic.
Suppose the wallet state is $\{5\}$ and the price is $79$.
Few people would hand the cashier the nickel in this situation, even though doing so would reduce the number of coins in their wallet after the transaction by $2$.
So let us consider a more realistic strategy.

\subsection*{The big spender}

If a spender does not have enough coins to cover the cost of their purchase and does not need to achieve the absolute minimum number of coins after the transaction, then the easiest course of action is to spend no coins.
In addition to assumptions \eqref{1}--\eqref{2}, let us therefore assume the following.
\begin{enumerate}
\setcounter{enumi}{2}
\item\label{3} If the spender does not have sufficient change to pay for the purchase, he spends no coins and receives change from the cashier.
\end{enumerate}

If the spender does have enough coins to cover the cost, it is reasonable to assume that he overpays as little as possible.
For example, if the wallet state is $\{25, 10, 5, 1, 1\}$ and the price is $13$ cents, then the spender spends $\{10, 5\}$.
How does a spender identify a subset of coins whose total is the smallest total that is greater than or equal to the purchase price?
Well, one way is to examine \emph{all} subsets of coins in the wallet and compute the total of each.
This naive algorithm may not be fast enough for the express lane, but it turns out to be fast enough to compute the transition matrix $M$ in a reasonable amount of time.

Now, there may be multiple subsets of coins in the wallet with the same minimal total.
For example, if the wallet state is $\{10, 5, 5, 5\}$ and the price is $15$ cents, there are two ways to make change.
Using the greedy algorithm as inspiration, let us assume the spender breaks ties by favoring bigger coins and spends $\{10, 5\}$ rather than $\{5, 5, 5\}$.
Therefore we adopt the following assumptions.
\begin{enumerate}
\setcounter{enumi}{3}
\item\label{4} If the spender has sufficient change, he makes the purchase by overpaying as little as possible and receives change if necessary.
\item\label{5} If there are multiple ways to overpay as little as possible, the spender favors $\{a_1, a_2, \dots, a_m\}$ over $\{b_1, b_2, \dots, b_n\}$, where $a_1 \geq a_2 \geq \dots \geq a_m$ and $b_1 \geq b_2 \geq \dots \geq b_n$, if $a_1 = b_1, a_2 = b_2, \dots, a_i = b_i$ and $a_{i+1} > b_{i+1}$ for some $i$.
\end{enumerate}

We refer to a spender who follows these rules as a \emph{big spender}.
Let's check that there are only finitely many states for a big spender's wallet.

\begin{lemma*}
Suppose a spender adheres to assumptions~\eqref{3} and \eqref{4}.
If the spender's wallet has at most $99$ cents before a transaction, then it has at most $99$ cents after the transaction.
\end{lemma*}

\begin{proof}
Let $0 \leq c \leq 99$ be the price, and let $n$ be the total value of coins in the spender's wallet.

If $c\leq n$, by \eqref{4}, the spender pays at least $c$ cents, receiving change if necessary, and ends up with $n - c$ cents after the transaction.  Since $n \leq 99$ and $c \geq 0$, we know that $n - c \leq 99$ as well.

If $c>n$,  since $n$ is not enough to pay $c$ cents, by \eqref{3} the spender only pays with bills, and receives $100-c$ in change, for a total of $n + 100 - c = 100 - (c - n)$ after the transaction.  Since $c > n$, we know that $c - n \geq 1$, so $100 - (c - n) \leq 99$.
\end{proof}

If a big spender begins with more than 99 cents in his wallet (because he did well at a slot machine), then he will spend coins until he has at most 99 cents, and then the lemma applies.  Thus any wallet state with more than 99 cents is only transient and has a long-term probability of 0.  Since there are finitely many ways to carry around at most 99 cents, the state space of the big spender's wallet is finite.

We are now ready to set up a Markov chain for the big spender.  The possible wallet states are the states totaling at most $99$ cents.
Each such state contains at most $3$ quarters, $9$ dimes, $19$ nickels, and $99$ pennies, and a quick computer filter shows that of these $4 \times 10 \times 20 \times 100 = 80000$ potential states only $6720$ contain at most $99$ cents.

To construct the $6720 \times 6720$ transition matrix for the big spender Markov chain, we must simulate all $6720 \times 100 = 672000$ possible transactions.
This is where the use of a computer becomes imperative.
Since we are using the naive algorithm, simulating this many transactions is somewhat time-consuming.
The authors' implementation took $8$ hours on a 2.6~GHz laptop.
The list of wallet states and the explicit transition matrix can be downloaded from the authors' web sites, along with a \textit{Mathematica} notebook containing the computations.

From the Perron--Frobenius theorem, we know that the limiting distribution $p$ exists.
However, computing it is another matter.
For matrices of this size, Gaussian elimination is slow.
If we don't care about the entries of $p$ as exact rational numbers but are content with approximations, it's much faster to use numerical methods.
\emph{Arnoldi iteration} is an efficient method for approximating the largest eigenvalues and associated eigenvectors of a matrix, without computing them all.
An implementation of Arnoldi iteration due to Lehoucq and Sorensen~\cite{Lehoucq--Sorensen} is available in the package ARPACK~\cite{ARPACK}, which is free to download and use.
This package is also used by \textit{Mathematica}~\cite{implementation}, so to compute the dominant eigenvector of a matrix one can simply evaluate
\[
	\texttt{Eigenvectors[N[Transpose[}matrix\texttt{]], 1]}
\]
in the Wolfram Language.
The symbol \texttt{N} converts rational entries in the matrix to floating-point numbers, and \texttt{Transpose} ensures that we get a left (not right) eigenvector.

ARPACK is quite fast.
Computing the dominant eigenvector for the big spender transition matrix takes less than a second.
And one finds that there are five most likely states, each with a probability of $0.01000$; they are the empty wallet $\{\}$ and the states consisting of $1$, $2$, $3$, or $4$ pennies.
Therefore $5\%$ of the time the big spender's wallet is in one of these states.

The expected number of coins in the big spender's wallet is approximately $10.05$.
This is more than twice the expected number of coins for the minimalist spender.
The expected numbers of quarters, dimes, nickels, and pennies are $1.06$, $1.15$, $0.91$, and $6.92$.
Assuming that all coin holders are big spenders (which is not actually the case, since cash registers dispense coins greedily), this implies that the distribution of coins in circulation is 10.6\%~quarters, 11.5\%~dimes, 9.1\%~nickels, and 68.9\%~pennies.
Compare this to the distribution of U.S. minted coins in 2014 ---
11.9\%~quarters, 17.4\%~dimes, 9.1\%~nickels, and 61.6\%~pennies.
Relative to the coin keeper strategy, the big spender distribution comes several times closer (as points in $\R^4$) to the U.S. mint distribution.

The expected total value of the big spender's wallet is computed to be $49.5$ cents.
This is the same value as for the minimalist spender, which may be surprising since the two spending strategies are so different.
However, it is a consequence of assumption~\eqref{1}, which specifies that prices are distributed uniformly.
If we ignore all information about the big spender's wallet state except its value, then we get a Markov chain with $100$ states, all equally likely, and the expected wallet value is $49.5$ cents.
Since the expected wallet value is preserved under the function which forgets about the particular partition of $n$, the big spender has the same expected wallet value.
In fact, any spending scenario in which the possible wallet values are all equally likely has an expected wallet value equal to the average of the possible wallet values.

\subsection*{The pennies-first big spender}

We have seen that while the minimalist spender carries $4.7$ coins on average, the big spender carries significantly more.
We can narrow the gap by spending pennies more intelligently.
For example, if the wallet state is $\{1, 1, 1, 1\}$ and the price is $99$ cents, then it is easy to see that spending the four pennies will result in fewer coins than not.

To determine which coins to pay with, the \emph{pennies-first big spender} first computes the price modulo $5$.
If he has enough pennies to cover this price, he hands those pennies to the cashier and subtracts them from the price.
Then he behaves as a big spender, paying for the modified price.

If the pennies-first big spender has fewer than $5$ pennies before a transaction, he has fewer than $5$ pennies after the transaction.
Therefore the pennies-first big spender never carries more than $4$ pennies, and the state space is reduced to only $1065$ states.
Computing the dominant eigenvector of the transition matrix shows that the expected number of coins is $5.74$.
This is only $1$ coin more than the minimum possible value, $4.7$.
So spending pennies first actually gets you quite close to the fewest coins on average.

The expected numbers of quarters, dimes, nickels, and pennies for the pennies-first big spender are $1.12$, $1.27$, $1.35$, and $2.00$.
This raises a question.
Is the expected number of pennies not just approximately $2$ but exactly $2$?
Imagine that the pennies-first big spender is actually two people, one who holds the pennies, and the other who holds the quarters, dimes, and nickels.
When presented with a price $c$ to pay, these two people can behave collectively as a pennies-first big spender without the penny holder receiving information from his partner.
If the penny holder can pay for $c \bmod 5$, then he does; if not, he receives $5 - (c \bmod 5)$ pennies from the cashier.
Since the penny holder doesn't need any information from his partner, all five possible states are equally likely, and the expected number of pennies is exactly $2$.

\section*{Other currencies}

The framework we have outlined is certainly applicable to other currencies.
We mention a few of interest, retaining assumptions \eqref{1}--\eqref{5}.

A \emph{penniless purchaser} is a spender who has no money.
Unless they get a job, their long-term wallet behavior is not difficult to analyze.
On the other hand, a \emph{pennyless purchaser} is a big spender who never carries pennies but does carry other coins.
Pennyless purchasers arise in at least two different ways.
Some governments prefer not to deal with pennies.
Canada, for example, stopped minting pennies as of 2012, so most transactions in Canada no longer involve pennies.
On the other hand, some people prefer not to deal with pennies and drop any they receive into the give-a-penny/take-a-penny tray.
Therefore prices for a pennyless purchaser are effectively rounded to a multiple of $5$ cents, and it suffices to consider $20$ prices rather than $100$.
Moreover, these $20$ prices occur with equal frequency as a consequence of assumption~\eqref{1}.
There are $213$ wallet states composed of quarters, dimes, and nickels that have value at most $99$ cents.
The expected numbers of quarters, dimes, and nickels for the pennyless purchaser are $1.12$, $1.27$, and $1.35$.

If these numbers look familiar, it is because they are the same numbers we computed for the pennies-first big spender!
Since we established that pennies can be modeled independently of the other coins for the pennies-first big spender, one might suspect that the pennies-first big spender can be decomposed into two independent components --- a pennyless purchaser (with $213$ states) and a penny holder (with $5$ states, all equally likely).
When presented with a price $c$ to pay, the pennyless component pays for $c - (c \bmod 5)$ as a big spender, receiving change in quarters, dimes, and nickels if necessary.
As before, if the penny holder can pay for $c \bmod 5$, then he does; if not, he receives $5 - (c \bmod 5)$ pennies in change.
Let us call the product of these independent components a \emph{pennies-separate big spender}.

However, this decomposition doesn't actually work.
For the pennies-\emph{first} big spender, if the price is $c = 1$ cent then the two wallet states $\{5\}$ and $\{5, 1\}$ result in different numbers of nickels after a transaction, so the pennyless component does in fact need information from the penny component.
Even worse, if the price is $c = 1$ cent and the wallet is $\{5\}$ then the pennies-\emph{separate} big spender's wallet becomes $\{5, 1, 1, 1, 1\}$, which is too much change!

Nonetheless, these two Markov chains are closely related.
Suppose $s_i$ and $s_j$ are two states such that some price $c$ causes $s_i$ to transition to $s_j$ for the pennies-first big spender.
If $s_i$ contains fewer than $c \bmod 5$ pennies, then the price $(c + 5) \bmod 100$ causes $s_i$ to transition to $s_j$ for the pennies-separate big spender; otherwise the price $c$ causes this transition.
Therefore the transition matrices for the pennies-first big spender and the pennies-separate big spender are equal, and this explains the numerical coincidence we observed.

Another spending strategy is the \emph{quarter hoarder}, used by college students and apartment dwellers who save their quarters for laundry.
All quarters they receive as change are immediately thrown into their laundry funds.
Of the $10 \times 20 \times 100 = 20000$ potential wallet states containing up to $9$ dimes, $19$ nickels, and $99$ pennies, there are $4125$ states for which the total is at most $99$ cents.
The expected number of coins for a big spender quarter hoarder is $13.74$, distributed as $1.60$ dimes, $1.21$ nickels, and $10.93$ pennies.

Finally, let's consider a currency no one actually uses.
Under assumptions~\eqref{1} and \eqref{2}, Shallit~\cite{Shallit} asked how to choose a currency so that cashiers return the fewest coins per transaction on average.
For a currency $d_1 > d_2 > d_3 > d_4$ with four denominations, he computed that the minimum possible value for the average number of coins per transaction is $389/100$, and one way to attain this minimum is with a $25$-cent piece, $18$-cent piece, $5$-cent piece, and $1$-cent piece.
So as our final model, we consider a fictional country that has adopted Shallit's suggestion of replacing the $10$-cent piece with an $18$-cent piece.
There are two properties of this currency that the U.S. currency does not have.
The first is that the greedy algorithm doesn't always make change using the fewest possible coins.
For example, to make $28$ cents the greedy algorithm gives $\{25, 1, 1, 1\}$, but you can do better with $\{18, 5, 5\}$.

The second property is that there is not always a unique way to make change using the fewest possible coins.
For example, $77$ cents can be given in five coins as $\{25, 25, 25, 1, 1\}$ or $\{18, 18, 18, 18, 5\}$.
The prices $82$ and $95$ also have multiple minimal representations.
(Bryant, Hamblin, and Jones~\cite{Bryant--Hamblin--Jones} give a characterization of currencies $d_1 > d_2 > d_3$ that avoid this property, but for more than three denominations no simple characterization is known.)
According to assumption~\eqref{5}, the big spender breaks ties between minimal representations of $77$, $82$, and $95$ by favoring bigger coins.
For example, the big spender spends $\{25, 25, 25, 1, 1\}$ rather than $\{18, 18, 18, 18, 5\}$ if both are possible.

The cashier doesn't care about getting rid of big coins, however.
So to make things interesting, let's refine assumption~\eqref{2} as follows.
\begin{enumerate}
\item[($2'$)] Cashiers return change using the fewest possible coins; when there are two ways to make change with fewest coins, the cashier uses each half the time.
\end{enumerate}
For example, a cashier makes change for $77$ cents as $\{25, 25, 25, 1, 1\}$ with probability $1/2$ and as $\{18, 18, 18, 18, 5\}$ with probability $1/2$.
Consequently, the transition matrix has some entries that are $1/200$.

For the minimalist spender in this currency, there are $100$ possible wallet states, and the expected number of coins is $\frac{1}{100} \sum_{i=1}^{100} |s_i| = 3.89$.
Note that this is the same computation used to determine the average number of coins per transaction.
In general these two quantities are the same, so reducing the number of coins per transaction is equivalent to reducing the number of coins in the minimalist spender's wallet.
Relative to the U.S. currency, the minimalist spender carries $0.81$ fewer coins in the Shallit currency.

The number of wallet states in the Shallit currency totaling at most $99$ cents is $4238$.
The pennies-first big spender strategy is not such a sensible way to spend coins, since if your wallet state is $\{18, 1, 1, 1\}$ and the price is $18$ cents then you don't want to spend pennies first.
For the big spender, however, the expected number of coins is $8.63$, so this currency also reduces the number of coins in the wallet of a big spender.
The expected numbers of quarters, $18$-cent pieces, nickels, and pennies are $0.66$, $0.98$, $2.10$, and $4.89$.

\section*{Cashing in}

In this paper, we have taken the question `What's in your wallet?'\ quite literally.
We determined the long-term behavior of wallets with various currencies under four spending strategies --- the coin keeper, the minimalist spender, the big spender, and the pennies-first big spender.
Numerical methods in linear algebra afford us access to this statistical information --- information that is arguably interesting to know and that may not be obtainable any simpler way.

However, while Arnoldi iteration allows us to quickly compute the dominant eigenvector of a transition matrix, the computation of the matrix itself can be time-consuming.
We used the naive algorithm for simulating a transaction, which looks at all subsets of a wallet to determine which subset to spend.
Is there a faster algorithm for computing, for example, the big spender's behavior?

There are many currencies we have not considered.
It would be interesting to know in which country of the world a big spender (or a smallest-denomination-first big spender) is expected to carry the fewest coins.
Or, which currency $d_1 > d_2 > d_3 > d_4$ of four denominations minimizes the expected number of coins in your wallet?

There are also spending strategies we have not considered, and indeed there are good reasons to vary some of the assumptions.
For example, assumption~\eqref{5} isn't universally true.
Given the choice between spending a dime or two nickels, the big spender spends the dime.  While the big spender minimizes the number of coins he spends, another spender might instead break ties by spending \emph{more} coins.
We could consider a \emph{heavy spender} who maximizes the number of coins spent from his wallet in a given transaction according to the following modification of assumption~\eqref{5}.
\begin{enumerate}
\item[($5'$)] \label{5prime} If there are multiple ways to overpay as little as possible, the spender favors $\{a_1, a_2, \dots, a_m\}$ over $\{b_1, b_2, \dots, b_n\}$ if $m > n$. 
\end{enumerate}
A heavy spender favors $\{5, 5\}$ over $\{10\}$.
Do assumptions~\eqref{3}, \eqref{4}, and (\hyperref[5prime]{$5'$}) completely determine the behavior of a heavy spender?
If so, how much lighter is the heavy spender's wallet?

Of course, the million-dollar question is whether typical people use any of these spending strategies.
How many coins is an actual person expected to carry?
Does a typical person have a consistent spending strategy, or does their behavior depend more on how much of a rush they're in?
If some people do use a consistent strategy, to what extent is the pennies-first big spender more realistic than the big spender?
On second thought, maybe it's just easier to use your credit card.

\end{document}